\documentclass[11]{siamltex}
\usepackage{amsmath,amssymb,amsfonts,graphicx,fancybox,shadow,color}
\usepackage{hyperref}

\newtheorem{conjecture}[theorem]{Conjecture}

\newcommand{\RR}{{\mathbb{R}}}
\newcommand{\CC}{{\mathbb{C}}}

\def\tu {\tilde{u}}
\def\tF {\tilde{F}}

\def\uu{{\mathbf u}}
\def\UU{{\mathbf U}}

\def\be#1\ee{\begin{equation}#1\end{equation}}

\begin{document}
\bibliographystyle{plain}

\pagestyle{myheadings}
%\markboth{V.~Druskin, V.~Simoncini and M.~Zaslavsky}{Adaptive rational Krylov space}

\title{  Lippmann-Schwinger-Lanczos algorithm for inverse scattering problems}
\author{
 V. Druskin\footnotemark[1],  S. Moskow\footnotemark[2] and M. Zaslavsky\footnotemark[3]}

\renewcommand{\thefootnote}{\fnsymbol{footnote}}

\footnotetext[1]{Worcester Polytechnic Institute, Department of Mathematical Sciences,
Stratton Hall,
100 Institute Road, Worcester MA, 01609 (vdruskin1@gmail.com)}
\footnotetext[2]{Department of Mathematics, Drexel University, Korman Center, 3141 Chestnut Street, Philadelphia, PA 19104
(moskow@math.drexel.edu)}
\footnotetext[3]{Schlumberger-Doll Research Center, 1 Hampshire St., 
Cambridge, MA 02139-1578 (mzaslavsky@slb.com)}

\maketitle

\begin{abstract} \  Data-driven reduced order models (ROMs) are combined with the  
Lippmann-Schwinger integral equation to produce a direct nonlinear inversion method. The ROM is viewed as a Galerkin projection and is sparse due to Lanczos orthogonalization. Embedding into the continuous problem, a data-driven internal solution is produced. This internal solution is then used in the Lippmann-Schwinger equation, thus making further iterative updates unnecessary. We show numerical experiments for spectral domain domain data for which our  inversion is far superior to the Born inversion and  works as well as when the true internal solution is known.  \end{abstract}

%%%%%%%%%%%%%%%%%%%%%%%%%%%%%%%%%%%%%%%%%%%%%%%%%%%
\section{Introduction}
%\cite{borcea2005continuum,borcea2014model,druskin2016direct,druskin2018nonlinear,borcea2017untangling,borcea2019robust, BoDrMaMoZa,druskin1999gaussian,druskin2002three}
This work extends the reduced order model (ROM) approach to inverse impedance, scattering and diffusion problems  that was developed   in the sequel of papers \cite{borcea2011resistor,borcea2014model,druskin2016direct,druskin2018nonlinear,borcea2017untangling,borcea2019robust,BoDrMaMoZa,borcea2020reduced}. 
This approach to solve multidimensional inverse problems using a ROM framework can be summarized by the following:
\begin{enumerate}
\item The boundary data, (discrete partial Dirichlet-to-Neumann maps as in \cite{borcea2011resistor} or their transient variant \cite{druskin2016direct,druskin2018nonlinear,borcea2017untangling,borcea2019robust, BoDrMaMoZa}) is matched by data-driven ROMs. The ROMs can be written as a tridiagonal or block-tridiagonal matrices  for one-dimensional and multi-dimensional problems respectively.  Data matching is performed via a direct layer stripping algorithm or a sequence of such algorithms.
\item The data-driven ROMs implicitly embed boundary data back into the interior as a discrete network. The network's coefficients are approximately localized averages of the PDE coefficients.  Approximate linear maps between these averages and the network coefficients can be found via so-called "optimal grid" or "`finite-difference Gaussian quadrature" \cite{borcea2011resistor,druskin2016direct}, 
or this map can be used as a preconditioner in optimization algorithms. \cite{borcea2014model,borcea2020reduced}.  
\end{enumerate}
In short, the main nonlinearity of the inverse problem is absorbed during the first stage thanks to the layer stripping, which is intrinsically nonlinear. This allows one to treat the transformed data as almost linear with respect to PDE coefficients. The layer stripping in essence is just the Euclidean polynomial division algorithm, which was further developed in the seminal works of Stieltjes and Lanczos, and later in the electrical circuit community. It also related to seminal works from the Soviet school (Marchenko, Gelfand, Levitan and Krein) on inverse spectral problems.  The origin of the embedding concept can be be traced to Krein \cite{kac1974spectral}, and to the idea of the spectrally matched second order staggered finite-difference grids first introduced in \cite{druskin1999gaussian}. These "optimal grids" or "finite-difference Gaussian quadrature"  rules provided a clear geometric interpretation of the data driven network's coefficients \cite{BoDr}.  
%
%Aside fro inverse problems, it has number of application, ranging from approximation of exterior PDE problems, e.g., \cite{druskin2016near}, to cluster analysis of big data sets on graphs \cite{druskin2018clustering}.

It is well known that data-driven ROMs can be rewritten in projection form, for example as a Galerkin system, with the same matrix for state variable realization in the interior, see for example \cite{doi:10.1137/1.9781611974829.ch7}. The relationship between these projections and finite-difference Gaussian quadrature rules was first studied in \cite{druskin2002three} and further developed in \cite{BoDrMaMoZa}. 
A critical property of the projection solution in the interior, which was first noticed in \cite{druskin2016direct}, is that the Galerkin basis corresponding the unknown coefficient is very close to the one from the homogeneous problem. 
This is a property particular to the sparse realization mentioned above, and would not be possible, for example, if one were using the full stiffness and mass matrices appearing in the Loewner product formulation normally used for data-driven ROM construction. 
Because of the sparse structure of of the system (tridiagonal finite-difference in 1D problems), the basis functions are localized and depend only weakly on the media perturbations. Rigorous analysis of this property can be obtained using the previously mentioned Marchenko, Gelfand, Levitan and Krein approach and is in progress at the moment \cite{BorceaZimmerling}.

Numerical experiments have demonstrated that the weak dependence of the basis functions on the unknown coefficient also holds in the multidimensional setting, and this gave rise to so-called nonlinear back-projection imaging algorithm \cite{druskin2018nonlinear}. 
The backprojection algorithm worked very well for time-domain wave problems where a larger system and higher resolution was possible, but failed in diffusion problems, allowing only small systems due to inherent ill-posedness. To overcome this problem, in \cite{BoDrMaMoZa} the authors with collaborators introduced the idea of using the projection formulation of the ROM to generate internal solutions from boundary data only. The approximate solution of inverse problem could then be computed by simply taking the Laplacian of the internal solution and dividing. This approach worked better than back projection for small data sets, but the derivatives amplified the approximation error of the internal solution and prevented its use for somewhat larger size systems. 

The main idea of this work is to use the data-generated internal solution $\tilde{u}_p$ (corresponding to unknown coefficient $p$) introduced in \cite{BoDrMaMoZa} in the Lippmann-Schwinger integral equation framework. If we consider the data $F_p$ for the unknown coefficient $p$ and the corresponding background data $F_0$ corresponding to a homogeneous or other known background, then 
the Lippmann-Schwinger integral equation with respect to the unknown $p$ can be written as \be\label{eq:LipSwi} F_p -F_0 =-\langle u_0,p u_p\rangle \ee
where $\langle ,\rangle$  is the continuous $L^2$ inner product on the PDE domain, and where $u_p, u_0$ are the unknown and background internal solutions respectively.
Because of the dependence of $u_p$ on $p$, equation (\ref{eq:LipSwi}) is generally nonlinear. The standard linearization is given by the famous Born approximation 
\be\label{eq:LipSwiBorn} F_p -F_0 \approx -\langle u_0,p u_0 \rangle, \ee
in which one assumes that  $u_p\approx u_0$, which is accurate only for small $p$. In this paper we suggest to use the more accurate approximation, the data generated internal solution ${\bf u}_p$.  This yields 
\be\label{eq:LipSwiL} F_p -F_0  \approx -\langle u_0,p {\bf u}_p\rangle.\ee
Indeed, the internal solution ${\bf u}_p$ depends on $p$, however, it can be precomputed directly from the data without knowing $p$. Consequently, (\ref{eq:LipSwiL}) is linear with respect to $p$. 
As we shall see, formula (\ref{eq:LipSwiL}) relies on the Lanczos algorithm for the computation of the ROM which is used to find ${\bf u}_p$, which is why we call (\ref{eq:LipSwiL}) the Lippmann-Schwinger-Lanczos equation.

As a by-product, the Lippmann-Schwinger-Lanczos approach resolves a number of problems arising in the data-driven ROM framework. Recently in \cite{borcea2020reduced}, the map between the ROM and $p$ was computed via iteration, and required multiple solutions of forward problems for different $p$. The Lippmann-Schwinger-Lanczos approach yields an explicit expression for this map. Furthermore, this approach also allows for a more rigorous derivation of the back-projection algorithm, and suggests natural extensions to more general data sets.

Finally, we should point out, that there are known approaches using internal solutions in the Lippmann-Schwinger framework, of which Marchenko redatuming is the closest to the suggested method here, for example, see \cite{Diekmann2020ImagingWT}. The main difficulty with Marchenko  redatuming lies in the accurate approximation of the inverse scattering transform in the continuous setting, and in the evaluation of some integrals. The Lanczos based ROM approach here can be interpreted as a linear-algebraic realization of the Marchenko-Gelfand-Levitan method \cite{druskin2016direct} with a data-driven spectral discretization, and as such, promises the best possible accuracy for a given data-set. As we shall see in our numerical experiments,  with as little as 5 Laplace frequencies and 8 source-receiver positions, our approach produces  results which are indistinguishable from the Lippmann-Schwinger formulation using the exact internal solutions.

This paper is organized as follows. In Section 2 we describe the entire process in detail for a one dimensional, single input single output (SISO) problem. We include the entire process here both for completeness of presentation and to extend \cite{BoDrMaMoZa} to complex data sets. This includes the construction of the ROM from the data, the Lanczos orthogonalization process, the generation of the internal solution and its use in the Lippmann-Schwinger equation. 
The generalization of this process to multiple input multiple output (MIMO) problems in higher dimensions is described in Section 3. Section 4 contains numerical experiments, and in the appendix we describe how the Lippmann-Schwinger Lanczos algorithm is related to other approaches to inversion using the data driven ROM.

\section{One dimensional SISO problem}
We begin this work with the one dimensional problem, since the presentation is simpler, and the ideas extend naturally to higher dimensions. In the first subsection we describe the problem setup and in the second subsection we show how one constructs the ROM from the data. In the third  and fourth subsections we discuss  tridiagonalization  of the ROM and generation of the internal solution from data only.  In the last subsection we show how to use the internal solutions in the Lippmann-Schwinger equation in order to solve the fully nonlinear inverse problem. 

  \subsection{Description of the SISO problem} We start by considering the single input single output (SISO) inverse problem in one dimension
	\be\label{eq:1D} -\frac{d^2u(x,\lambda)}{dx^2} +p(x)u(x,\lambda)+\lambda u(x,\lambda)=g(x) \quad \frac{du}{dx}|_{x=0}=0,\ \frac{du}{dx}|_{x=L}=0,
	\ee
	where $0<L\le\infty$. The source $g(x)$ is assumed to be a compactly supported real distribution localized near the origin, for example, roughly speaking, $ g = \delta(x-\epsilon)$ with small $\epsilon>0$. (Note that when $g$ is a delta function at the origin this corresponds to an inhomogeneous Neumann boundary condition.) We can write the solution formally as
	\be\label{eq:resolvent} u= \left(-\frac{d^2}{dx^2} +{p}I+\lambda I\right)^{-1}g \ee
	where the inverse is understood to correspond to homogeneous Neumann boundary conditions. 
Consider  $\lambda_j\in \CC\setminus \RR_-$, $j=1,\ldots, m$, with $\Im \lambda_j\ge 0$.  { We note that for nonnegative $p$ the above resolvent is well defined for $\lambda$ off of the negative real axis. }
The SISO transfer function is then 
\be\label{eq:transfer}
	F(\lambda)=\int_0^L g(x) u(x,\lambda)dx =\langle g,u\rangle =\langle g,\left(-\frac{d^2}{dx^2} +{p}I+\lambda I\right)^{-1}g\rangle \ee
where throughout the paper we use $\langle, \rangle$ to denote the continuous Hermitian $L^2$ inner product $$\langle w,v\rangle=\int_0^L \bar{w}(x)v(x)dx ,$$ which in the one dimensional case is $L^2(0,L)$.
For $2m$ real data points, that is, for $\Im \lambda_j=0$,  we consider the data 
\begin{equation} \label{realdata} F(\lambda)|_{\lambda=\lambda_j}\in\RR, \ \ \ \frac{dF(\lambda)}{d\lambda}|_{\lambda=\lambda_j}\in \RR \ \ \ \mbox{for} \ \  j=1,\ldots, m.\end{equation}  For complex data points, we consider just \begin{equation}\label{complexdata} F(\lambda)|_{\lambda=\lambda_j}\in\CC \ \ \ \mbox{for} \ \  j=1,\ldots, m.\end{equation} Note that the complex case is equivalent to also having data at $\bar \lambda_j$, since $F(\overline{\lambda})=\overline{F}(\lambda)$ from the fact that $g$ and $p$ are real. From this one can see that for $\lambda_j$ close to real,  
\begin{eqnarray} \frac{dF(\lambda)}{d\lambda}|_{\lambda=\Re{\lambda_j}} &=& \lim_{\Im\lambda_j\to 0} {F(\lambda_j)-F(\overline{\lambda_j}) \over{ \lambda_j -\overline{\lambda_j}}} \nonumber \\  &=& \lim_{\Im\lambda_j\to 0} { \Im F(\lambda_j)\over {\Im{\lambda_j}}} \end{eqnarray} and hence the real data can be viewed as the natural extension of complex data.  The SISO inverse problem is then to determine $p(x)$ in (\ref{eq:1D}) from the data (\ref{complexdata}) or (\ref{realdata}).

\subsection{Construction of the data-driven ROM}

We will treat $u(x,\lambda_j)$ as a continuous basis function (or it can be viewed as an infinite dimensional vector column) and consider the projection subspace $$\mathbb{U}=\mbox{span}\{u_1(x)=u(x,\lambda_1),\ldots , u_m(x)=u(x,\lambda_m)\}.$$  We define the data-driven ROM as the Galerkin system for this subspace \be\label{eq:ROM}
Sc(\lambda)+\lambda Mc(\lambda)=b
\ee
where $S,M\in\CC^{m\times m}$ are Hermitian positive definite matrices with the stiffness matrix $S$ given by $$S_{ij}=\langle {u}'_i,u'_j\rangle +\langle p{u}_i,u_j\rangle $$ and mass matrix $M$ given by  $$M_{ij}=\langle {u}_i,u_j\rangle. $$ The right hand side $b\in\CC^m$ is a column vector with components $$b_j =\langle  {u}_j, g  \rangle, $$ and the Galerkin solution for the system is determined by the vector valued function of $\lambda$,  $c(\lambda)\in\CC^{m}$.  Note that $c(\lambda)$ corresponds to a column vector of coefficients of the solution with respect to the above basis of exact solutions. The matrices $S$ and $M$ are obtained from matching conditions that ensure that the ROM transfer function  \[ \tF(\lambda)=b^*c\]
matches the data (here and below $^*$ denotes conjugate transpose). For real $\lambda_i$ this is the same as in \cite{BoDrMaMoZa}. For complex $\{\lambda_j\}^m_{j=1}$ multiplying (\ref{eq:1D}) for $\lambda=\lambda_i$ by $\bar{u}_j$ in the inner product $\langle\cdot ,\cdot \rangle$ and integrating by parts we obtain
\begin{equation}
\label{eq:loew1}
\bar{S}_{ij}+\lambda_i\bar{M}_{ij}=\bar{F}(\lambda_j).
\end{equation}
Similarly, multiplying the conjugate of (\ref{eq:1D}) for $\lambda=\bar{\lambda}_j$ by $u_i$ in the inner product $\langle\cdot ,\cdot \rangle$, we obtain
\begin{equation}
\label{eq:loew2}
S_{ij}+\lambda_jM_{ij}=b_i,
\end{equation}
where $b_i=\bar{F}(\lambda_i)$.
Subtracting the complex conjugate of (\ref{eq:loew1}) from (\ref{eq:loew2}) we obtain the expression for the mass matrix
\begin{equation}
\label{eq:massmtr}
M_{ij}=\frac{\bar{F}(\lambda_i)-F(\lambda_j)}{\lambda_j-\bar{\lambda}_i}.
\end{equation}
Similarly, the elements of stiffness matrix have the form 
\begin{equation}
\label{eq:stifmtr}
S_{ij}=\frac{F(\lambda_j)\lambda_j-\bar{F}(\lambda_i)\bar{\lambda}_i}{\lambda_j-\bar{\lambda}_i}.
\end{equation}
Furthermore, the solution to (\ref{eq:1D}) is close to its Galerkin projection 
\[u(\lambda)\approx \tu(\lambda)={V}c(\lambda)={V} (S+\lambda M)^{-1}b \]
where  ${V}$  represents the row vector of basis functions $u_i$,
$${ V} = [ u_1 ,\ldots , u_m ] .$$   The following proposition says that for $S$ and $M$ given above the ROM transfer function matches the data. Its proof follows immediately from the fact that the exact solutions are in the trial space and the fact the $F$ commutes with the complex conjugate, that is,  $\bar{\tilde{F}}(\lambda)= \tilde{F}(\bar{\lambda})$.
\begin{proposition}\label{prop:1}
Assume that $\Im \lambda_j\ne 0$ for $j=1,\ldots,m$. 
Then the Galerkin projection of the solution of (\ref{eq:1D})
  \[\tu(\lambda)=Vc(\lambda)=V(S+\lambda M)^{-1}b\] 
is exact at $\lambda=\lambda_j$ ,
$$ \tu(\lambda_j) = u(\lambda_j),$$
and hence $$ \tF(\lambda_j)=b^*(S+\lambda_j M)^{-1}b=F(\lambda_j)$$ and $$  \tF(\bar\lambda_j)=b^*(S+\bar\lambda_j M)^{-1}b=F(\bar\lambda_j) $$
for $j=1,\ldots, m$.
\end{proposition}

The corresponding proposition for $\lambda_j$ real was shown in \cite{BoDrMaMoZa}.

\subsection{Lanczos tridiagonalization}
In the next step, we orthogonalize the above basis of exact solutions by using the Lanczos algorithm. More precisely, 	we run $m$ steps of the { $M$-Hermitian}  Lanczos algorithm corresponding to  matrix $A=M^{-1}S$ and initial vector  $M^{-1}b$. This yields tridiagonal matrix $T\in\RR^{m\times m}$ and $M$-orthonormal Lanczos vectors $q_i\in \CC^m$, such that
	\be\label{eq:lancz} AQ =Q T , \qquad Q^*MQ=I,\ee
where $$Q=[q_1, q_2, \ldots, q_m]\in{\CC^{m\times m}},$$ and $$q_1=M^{-1}b/\sqrt{b^*M^{-1}b}.$$ This resulting new basis of continuous functions will be orthonormal with respect to the continuous $L^2(0,L)$ inner product and are given by the vector $${VQ=  [ \sum_{j=1}^mq_{j1}u_j ,\ldots , \sum_{j=1}^m q_{jm}u_j ]}  .$$   The Galerkin projection of (\ref{eq:1D}) and its transfer function can then be written in Lanczos coordinates as
	\be \label{eq:state} \tu(\lambda)=\sqrt{b^*M^{-1}b}V Q(T+\lambda I)^{-1}e_1, \ee \be \tF(\lambda)= (b^*M^{-1}b)  e_1^*(T+\lambda I)^{-1}e_1\ee
where $e_1 = (1,0,\ldots,0)^T $ is the first coordinate column vector in $\RR^m$. 
It is known, see for example \cite{St}, that when $\Sigma\in\CC\setminus \RR_-$ is compact, for any sequence  $\lambda_j\in \Sigma$, $j=1,\ldots,m$  we have that for any $\lambda\in\Sigma$ the corresponding Galerkin solution converges exponentially $\tu\rightarrow u$ as $m\rightarrow\infty$ with a uniform linear rate. 

\subsection{Internal solutions}
We assume that we do not know $p$,  yet we want to compute $u(x,\lambda)$ directly from the data. Recall that the output of the Lanczos algorithm, tridiagonal $T$ and change of basis $Q$, were obtained from data only. However, we do not know the original basis of exact solutions $V$.  We propose to approximate $u$ internally, as was done in \cite{BoDrMaMoZa}, by replacing the unknown orthogonalized internal solutions $VQ$ with orthogonalized background solutions $V_0Q_0$ corresponding to background $p_0=0$. Here $V_0$ is the row vector containing the basis of background solutions 
$${ V_0} = [ u^0_1,\ldots , u^0_m ] $$ 
to (\ref{eq:1D}) corresponding to $p=p_0=0$ and the same spectral points $\lambda= \lambda_1, \ldots \lambda_m$, and  $Q_0$ is computed from Lanczos orthogonalization of the background ROM.  That is, one can compute an approximation to $u(x,\lambda)$ using
$$ u \approx \sqrt{b^*M^{-1}b}V_0 Q_0(T+\lambda I)^{-1}e_1 ,$$
which is obtained from data only. 
\begin{conjecture}\label{conj:1} Assume that $V_0$ and $Q_0$, are the solution basis and Lanczos matrix  corresponding to the background $p_0=0$. 
Then for all  $\lambda\in\Sigma$ we have that 
\be\label{eq:intern} u=\lim_{m\to\infty}\uu= \sqrt{b^*M^{-1}b}V_0 Q_0(T+\lambda I)^{-1}e_1.
	\ee
that is, for any fixed $\lambda\in \Sigma$, as the number of data points in $\Sigma$ approaches infinity, the data generated internal solution $\uu$ converges to the true solution $u$. 
\end{conjecture}
	
\subsection{Nonlinear inverse problem} We now use the Lippmann-Schwinger formulation for the solutions $u$ to solve the nonlinear inverse problem by using the internal solutions described above. 
From  (\ref{eq:transfer})  we obtain
\begin{eqnarray}F_0(\lambda_j)-F(\lambda_j)=\langle g,\left(-\frac{d^2}{dx^2} +\lambda I\right)^{-1}g\rangle-\langle g,\left(-\frac{d^2}{dx^2} +{p}I+\lambda I\right)^{-1}g\rangle \nonumber \\=\langle \left(-\frac{d^2}{dx^2} +{p}I+\bar{\lambda} I\right)^{-1}g, p\left(-\frac{d^2}{dx^2} +\lambda I\right)^{-1}g\rangle 
\label{eq:operator}\end{eqnarray}
that is,
	\be  \label{eq:int}F_0(\lambda_j)-F(\lambda_j)=\int \bar{u}_0(x,\lambda_j ) u(x,\lambda_j )p(x)dx, \\
	\qquad j=1,\ldots, m \ee
which can also be seen as a direct consequence of the usual Lippmann-Schwinger formulation with the Green's function.  Here $u_0$ corresponds to the solution to the background problem with $p_0=0$.
If all $\lambda_j$ are complex, then (\ref{eq:int}) yields $2m$ real equations for $p$, which is the same as the number of data points.
For real $\lambda_j\in \RR$ we also have $2m$ real equations 
\begin{eqnarray}\label{eq:intd}F_0(\lambda_j) -F(\lambda_j) =\langle u_0,{p}u\rangle=\int \bar{u}_0(x,\lambda_j ) u(x,\lambda_j )p(x)dx, \\  \frac{d}{d\lambda}(F_0-F)|_{\lambda= \lambda_j}=\int \frac{d}{d\lambda}[\bar{u}_0(x,\lambda )u(x,\lambda_j )]_{\lambda=\lambda_j} p(x)d  x\nonumber \end{eqnarray}
for $j=1,\ldots,m$. Of course, the internal solutions $u(x,\lambda_j )$ and their derivatives with respect to $\lambda$ are unknown, and they depend on $p$, so the system (\ref{eq:int}-\ref{eq:intd}) is nonlinear with respect to $p$. A Born linearization replaces $u_0$ with $u$, however, it is only accurate for small $p$.
Using Conjecture~\ref{conj:1}, we replace $u(x,\lambda )$  in (\ref{eq:int}-\ref{eq:intd}) with its approximation $$ u \approx \uu=\sqrt{b^*M^{-1}b}V_0 Q_0 (T+\lambda I)^{-1}e_1.$$ 
For example, for the case of complex $\lambda_j$ we can write the new system for $p$ as
\be\label{eq:oper}
\delta F= \int W(x) p(x) dx
\ee
where $$\delta F=[F_0(\lambda_1)-F(\lambda_1),\ldots, F_0(\lambda_m)-F(\lambda_m)]\in\CC^{m},$$  and     $$W=[\uu(x,\lambda_1)\overline{u}_0(x,\lambda_1),\ldots,
\uu(x,\lambda_m) \overline{u}_0(x,\lambda_m)]$$ 
is an $m$-dimensional vector of complex valued functions on $(0,L)$.  By construction, $\uu$  can be directly computed from the data without knowing $p $, by using (\ref{eq:intern}), thus making nonlinear system (\ref{eq:oper})  linear! A reasonable regularization of (\ref{eq:oper}) is to restrict $p$ to the dominant left singular vectors of  $W$. We will refer to (\ref{eq:oper}) as a Lippmann-Schwinger-Lanczos system.

\section{Multidimensional MIMO problem}
\subsection{Description of the MIMO problem}
We consider the boundary value problem on $\Omega\in\RR^d$ for
\be\label{d-Schrod}-\Delta u^{(r)} +p u^{(r)} +\lambda u=g^{(r)}, \quad  \frac{d u}{d\nu}\large|_{\partial \Omega}=0, \  r=1,\ldots, K,\ee
where $g^{(r)}$ are localized  sources, e.g., boundary charge distributions,  supported near or at an accessible part $S$ of $\partial \Omega$.
Let $$G=[g^{(1)}, g^{(2)}, \ldots, g^{(K)}]$$  and $$U=[u^{(1)}, u^{(2)}, \ldots, u^{(K)}]$$
again understood as vectors of continuous source and solution functions respectively. Then the multiple-input multiple output (MIMO) transfer function is a matrix valued function of $\lambda$
	\be\label{eq:transferMIMO}
	F(\lambda)= \langle G,U\rangle =\langle G,\left(-\frac{d^2}{dx^2} +\diag{p}+\lambda I\right)^{-1}G\rangle\in\CC^{K\times K}
	\ee
where $\langle, \rangle$ again represents the continuous Hermitian $L^2(\Omega)$ inner product, that is,  here 
$$\langle G,U\rangle = \int_\Omega G^* U dx, $$ matrix valued in this case.  For real positive $\lambda$ we will have symmetric positive definite $F(\lambda)$. As in the SISO case, we consider the inverse problem with data given by $2m$ real symmetric $K\times K$ matrices, that is,  for $\Im \lambda_j=0$ our data are $$F(\lambda)|_{\lambda=\lambda_j}\in\RR,$$ and  $$\frac{F(\lambda)}{d\lambda}|_{\lambda=\lambda_j}\in \RR,$$
and $F(\lambda)|_{\lambda=\lambda_j}\in\CC$ otherwise.

\subsection{Construction of the MIMO data-driven ROM}
We consider the $mK$ dimensional projection subspace $$\mathbb{U}=\mbox{span}\{U_1(x)=U(x,\lambda_1),\ldots , U_m(x)=U(x,\lambda_m)\}.$$ We then define the MIMO data-driven ROM as \be\label{eq:ROMMIMO}
(S+\lambda M)C(\lambda) = B
\ee
where $S,M\in\CC^{mK\times mK}$ are Hermitian positive definite matrices,  $B\in\RR^{mK \times K}$,  and $C\in\CC^{mK\times K} $ is a matrix valued function of $\lambda$, again corresponding to coefficients of the solution with respect to the above basis of exact solutions. Stiffness and mass matrices can be written as block extensions of their counterparts in (\ref{eq:ROM}), with blocks given by  $$S=(S_{ij}=\langle {U'}_i,U'_j\rangle)+\langle p{U}_i,U_j\rangle)$$ and $$M=(M_{ij}=\langle {U}_i,U_j\rangle).$$  Once again $S$ and $M$ are obtained from imposing the conditions that the ROM transfer function  \[ \tF(\lambda)=B^*C(\lambda) \]
matches the data.

\subsection{Block-Lanczos tridiagonalization}
The next step is to run $m$ steps of the {$M$-Hermitian}  block-Lanczos algorithm with matrix $$A=M^{-1}S$$ and initial block  vector $M^{-1}B$.  From this we obtain the {Hermitian} block-tridiagonal matrix {$$T\in\CC^{Km\times Km}$$} with  $K\times K$ blocks,  and $M$-orthonormal Lanczos block-vectors $q_i\in \CC^{mK\times K}$. From this we obtain the block counterpart of (\ref{eq:lancz}) 
	$$ Q=[q_1, q_2,\ldots, q_m]\in{\CC^{mK\times mK}},$$ where $$ q_1=M^{-1}B(B^*M^{-1}B)^{-1/2}.$$
The state solution, that is, the Galerkin projection of the true solution, can then be written in Lanczos coordinates as
	\be \label{eq:stateB} \tilde{U}(\lambda)=\sqrt{B^*M^{-1}B}V Q(T+\lambda I)^{-1}E_1,\ee
where $E_1\in\RR^{mK\times K}$ consists of the first $K$ columns of identity matrix $I\in\RR^{mK\times mK}$.
{ Using block generalization of the SISO case, we express the orthogonalized basis as 
$$VQ=  [ \sum_{j=1}^m U_jq_{j1} ,\ldots , \sum_{j=1}^m U_jq_{jm}] $$
where $\left\{q_{ji}\in \CC^{K\times K}\right\}^{m}_{i,j=1}$ are the blocks of matrix $Q$ and $U_j=[u_{1j}(x),\ldots u_{Kj}(x)]$ are the corresponding vectors of solutions.} Note that this basis depends on the unknown internal solutions.
\subsection{Internal solutions}
As in the MIMO case, we replace the basis $VQ$ in (\ref{eq:stateB}) with its background counterpart $V_0Q_0$ to obtain an internal derived entirely from data. 
The SISO result is generalizable to MIMO case, and can be vaguely formulated as following conjecture
\begin{conjecture}\label{conj:1MIMO}
	Assume we have  sequence of  input sets $$G_k=[g_k^{(1)},g_k^{(2)}, \ldots, g_k^{(K)}],$$  such that as $k\rightarrow \infty$ the range of $G_k$ approximates surface distributions at the observable boundary  $S\subset\partial \Omega$.  Then for all $ \lambda\in\Sigma$
	\be\label{eq:internMIMO} U(x,\lambda) =\lim_{m,K\to\infty}\UU := \sqrt{B^*M^{-1}B}V_0 Q_0(T+\lambda I)^{-1}E_1.
	\ee
\end{conjecture}
	
\subsection{Nonlinear MIMO inverse problem}
From  (\ref{eq:transfer}) we obtain
\begin{multline} F_0(\lambda_j)-F(\lambda_j)=\langle G,\left(-\Delta +\lambda I\right)^{-1}G\rangle-\langle G,\left(-\Delta +pI+\lambda I\right)^{-1}G\rangle  \\ =\langle \left(-\Delta+{p}I+\lambda I\right)^{-1}G, p\left(-\Delta +\lambda I\right)^{-1}G\rangle \\ =\int U^*_0(x,\lambda_j ) p(x)U(x,\lambda_j )dx, \nonumber\end{multline}
that is,  \be\label{eq:intMIMO}F_0(\lambda_j)-F(\lambda_j)=\int U^*_0(x,\lambda_j ) p(x)U(x,\lambda_j )dx, \qquad j=1,\ldots, m.\ee
Here again the subscript $0$ corresponds to background solutions with $p=0$.
If all $\lambda_j$ are complex, then (\ref{eq:int}) gives $2m$ real equations for $p$, equal to the number of data points. 
For $\lambda_j\in \RR$ we will have
\begin{eqnarray} \label{eq:intdMIMO}(F_0-F)|_{\lambda_j} &=& \int U^*_0(x,\lambda_j) p(x) {U}(x,\lambda_j )dx, \nonumber \\  \frac{d}{d\lambda}(F_0-F)|_{\lambda_j}  &=& \int \frac{d}{d\lambda}(U^*_0(x,\lambda) {U}(x,\lambda))|_{\lambda=\lambda_j} p(x)d x.\end{eqnarray}
Similar to the SISO case, 
by using Conjecture~\ref{conj:1MIMO} we replace $U(x,\lambda )$ with its approximation $\UU$ in (\ref{eq:intMIMO}) and (\ref{eq:intdMIMO}). 
Again, like in the SISO case, precomputing $\UU$ via the data-driven algorithm (\ref{eq:internMIMO}) will yield the  linear system for $p$:
\begin{eqnarray} \label{eq:intdMIMO2}(F_0-F)|_{\lambda_j} &=& \int U^*_0(x,\lambda_j ) p(x) \mathbf{U}(x,\lambda_j )dx, \nonumber \\  \frac{d}{d\lambda}(F_0-F)|_{\lambda_j}  &=& \int \frac{d}{d\lambda}(U^*_0(x,\lambda )\mathbf{U}(x,\lambda))|_{\lambda=\lambda_j} p(x)d x.\end{eqnarray}

\section{Numerical Experiments}

In this section we present numerical results for reconstructing a 2D two-bump media. For simplicity, we considered noiseless case only. In the presence of noise, constructing the data-driven ROM may become unstable and requires regularization. This problem was resolved in \cite{borcea2019robust} by optimal SVD-based truncation of the unstable part.  However, to avoid additional complications, we skipped that part in this work. {In our experiments we chose such frequencies that provide enough sensitivity for both near and far zones.} In the first experiment we considered two low contrast bump inclusions (see Fig.\ref{fig:num1}, top left). The measurement setup mimics one from medical imaging, i.e. we had $K=8$ sources located on the boundary, two on each side (see red crosses at Fig.\ref{fig:num1}, top left). {We used $m=5$ positive spectral values $\lambda=1,2,14,50,128$. The forward problem was discretized on a regular triangular grid using finite-element method with $N=151\times 151$ elements. We approximated (\ref{eq:intdMIMO2}) using nodal quadrature on a $301\times 301$ grid. Here and below, to compute discrete $p$, the obtained ill-conditioned system was solved by projecting it onto its dominant eigenvectors.} On the top right of Fig.\ref{fig:num1} we plotted the reconstruction when actual true internal solution ${U}(x,\lambda))$ is available. {We call  this scenario 'Cheated IE',  and use it to emulate the  best possible outcome of hypothetical  iterative (aka distorted  Lippmann-Schwinger ) realizations,  when at the terminal stage the true $p$ is approximately known.} The usual Born linearization result (when we replace ${U}(x,\lambda))$ in (\ref{eq:intdMIMO}) with ${U}_0(x,\lambda))$) is shown on bottom left of Fig.\ref{fig:num1}. Finally, the reconstruction using our approach is plotted on the bottom right of Fig.\ref{fig:num1}. As one can observe, the Born approximation results in multiple artifacts and produces a significantly worse reconstruction compared to our approach and 'Cheated IE'.

In our second experiment we considered the same two bumps but with increased contrast (see Fig.\ref{fig:num2},  top left). The source locations were the same as in the previous experiment, however, we we added one negative frequency: $\lambda=-24,1,2,14,50,128$, so $m=6$. Forward problem and (\ref{eq:intdMIMO2}) discretizations were the same as in the previous example. Similar to our first experiments we plotted  the results obtained using the 'Cheated IE', the Born approximation and our approach (see on the top right, bottom left and bottom right of Fig.\ref{fig:num2}, respectively). Because of higher contrast, the Born linearization failed to produce a meaningful reconstruction. However, both the 'Cheated IE' and our approach still performed well. 

\begin{figure}[htb]
\centering
\includegraphics[scale=0.6]{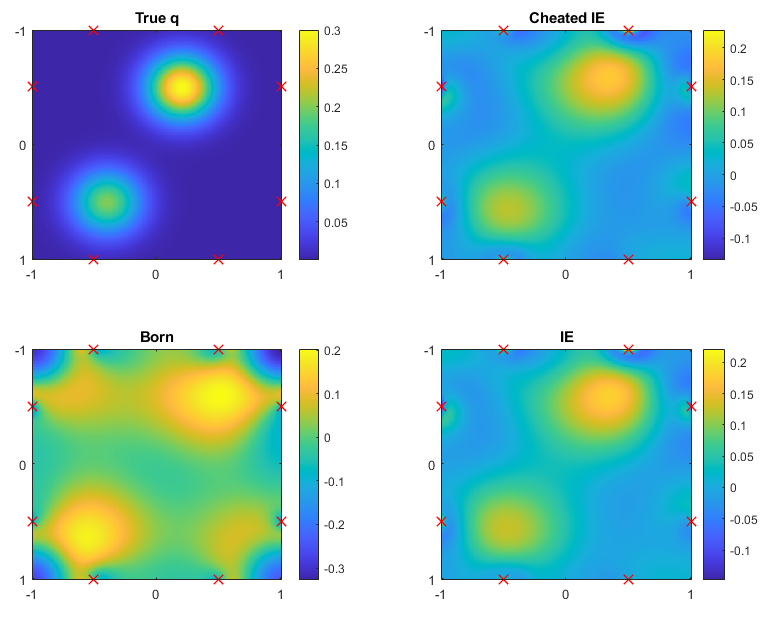} 
\caption{Experiment 1: True medium (top left) and its reconstructions using 'Cheated IE' (top right), Born linearization (bottom left) and our approach (bottom right)}
\label{fig:num1}
\end{figure}

%
%In our second experiment we considered the same two bumps but with increased contrast (see Fig.\ref{fig:num2},  top left). The sources locations were the same, however out of $s=6$ spectral values we took one negative and 5 positive. Similar to our first experiments we plotted results obtained using 'Cheated IE', Born approximation and our approach (see on the top right, bottom left and bottom right of Fig.\ref{fig:num2}, respectively). Because of higher contrast, Born linearization failed to produce meaningful reconstruction, however both 'Cheated IE' and our approach still performed well. 
%
\begin{figure}[htb]
\centering
\includegraphics[scale=0.6]{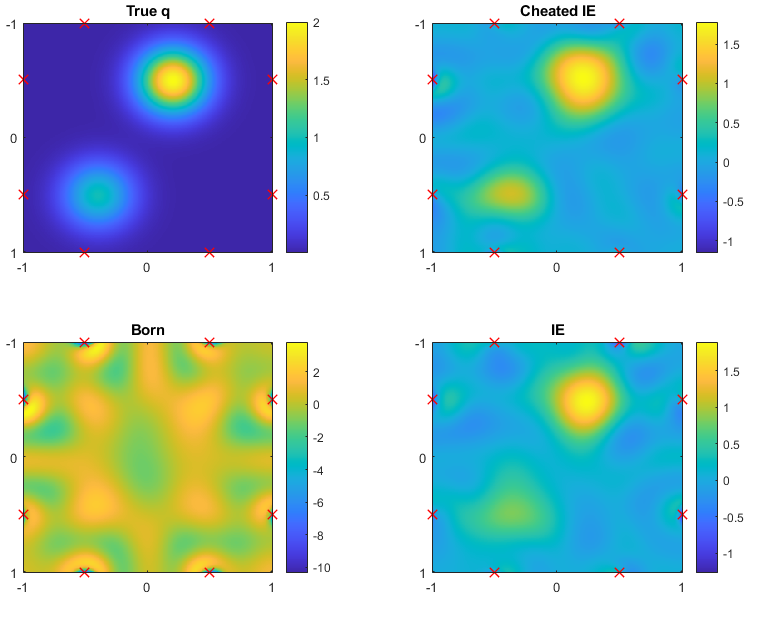} 
\caption{Experiment 2: True medium (top left) and its reconstructions using 'Cheated IE' (top right), Born linearization (bottom left) and our approach (bottom right)}
\label{fig:num2}
\end{figure}

In our third experiment we again considered a medium with 2 bumps with increased contrast (see Fig.\ref{fig:num3},  top left), however, with data acquisition similar to surface geophysics or radars.  That is, we probed the medium with $K=12$ sources located on one (upper) side $y=-1$ only (see red crosses at Fig.\ref{fig:num3}, top left). For better aperture, the lateral extent of the acquisition boundary was three times larger 
than the depth of the domain. {In this example we used the same $m=5$ positive frequencies as in the first experiment: $\lambda=-24,1,2,14,50,128$. The number of finite elements was $N=451\times 151$, and (\ref{eq:intdMIMO2}) was approximated using nodal quadrature on a $901\times 301$ grid.} We compared the results of the 'Cheated IE', the Born approximation  and our approach (see on the top right, bottom left and bottom right of Fig.\ref{fig:num3}, respectively). Similar to the previous 
example, the reconstruction via the Born approximant is totally unsatisfactory because of high contrast of the inclusions. At the same time, our approach, along with 'Cheated IE', performs well. 

\begin{figure}[htb]
\centering
\includegraphics[scale=0.6]{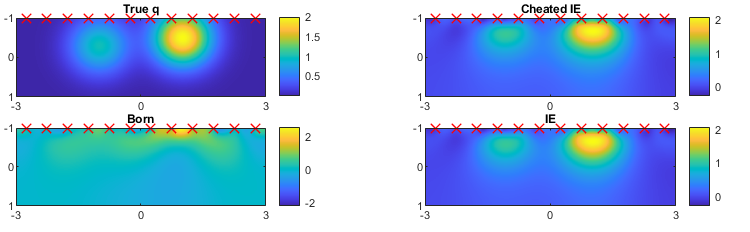} 
\caption{Experiment 3: True medium (top left) and its reconstructions using 'Cheated IE' (top right), Born linearization (bottom left) and our approach (bottom right)}
\label{fig:num3}
\end{figure}
%Finally, in our fourth experiments we considered the same setup as in the previous one but reduced the number of sources/receivers to $K=2$. This is the minimum to obtain meaningful medium reconstruction, though, as expected, the results are significantly worse than for $K=12$. However, as one can observe, even for this example our approach and 'Cheated IE' produced better image than Born (see Fig.\ref{fig:num4}).
%\begin{figure}[htb]
%\centering
%\includegraphics[scale=0.6]{0negfr_ar3to1_2src_2rcv_lr_cl_obj_2} 
%\caption{Experiment 4: True medium (top left) and its reconstructions using 'Cheated IE' (top right), Born linearization (bottom left) and our approach (bottom right)}
%\label{fig:num4}
%\end{figure}
\section{Conclusion and discussion}
{We developed a direct algorithm to solve inverse scattering problems that we call the  Lippmann-Schwinger-Lanczos (LSL) algorithm.  The main novelty of this algorithm is plugging an approximate internal solution computed via the data-driven Lanczos algorithm into the Lippmann-Schwinger system.  This adds some insignificant computational  expense compared to the simplest linearized Born approach, yet demonstrates significant improvement of the inversion quality for the problems where Born underperforms or even fails. Moreover,  in our  numerical experiments the inversion results  are practically indistinguishable from the ones emulating the best possible scenario for hypothetical iterative algorithms. Importantly, the data-driven ROM formulation allows one to obtain qualitatively good images with rather sparse amounts of data.

Indeed, we understand that there certainly can be more complicated situations and/or insufficient data, where our approximate internal solution may not be accurate enough for such a favorable outcome. In this case, however, the LSL can be used in an iterative loop, by updating the background and the internal solutions, with hypothetical   quadratic convergence. We plan to investigate such an approach in our future work.

Another direction for future research should include the extension of the LSL algorithm to first order Schrodinger  formulations, which are generally preferable for problems with sharp discontinuities \cite{borcea2020reduced}. The first order formulation gives a clear generalization path to elasticity and Maxwell's systems  \cite{borcea2019robust}, that we plan to investigate. Another direction of research will be application to the LSL of the ROM regularization from \cite{borcea2019robust}.

We also sketched in the appendix theoretical connections of the LSL algorithm to recent ROM based algorithms such as nonlinear back-projection \cite{druskin2018nonlinear} and ROM-based optimization \cite{borcea2020reduced}, which will hopefully give some new insight into these approaches.}
\appendix
\section{Generalizations and extensions}
In this appendix we show how the system (\ref{eq:oper}) is connected to other algorithms for extracting the unknown $p$ from the ROM, and describe how it lends itself naturally to extensions to other data sets. 
\subsection{Connection  to the back-projection algorithm}
The back-projection algorithm is based on the idea that the perturbed and background operators differ only in the lower order term, so that their difference should approximate $p$. Using the reduced models and replacing $VQ$ with $V_0Q_0$, the difference $$A_p-A_0= pI $$ is approximated by the operator $\mathcal{P}$
\be \label{eq:backprim2}
A_p -A_0 \approx \mathcal{P} \ee
where
 \be \label{eq:backprim}
\mathcal{P} w  =V_0 {Q}_0 (T-T_0)\int   (V_0 {Q}_0)^* w .\ee 
 Recall that $V_0Q_0$ is the row vector of orthogonalized background solutions.
  Here $T$ and $T_0$ are obtained via data matching conditions, for example via time-domain matching, or the frequency matching as described above. 
  
     To see how this relates to Lippmann Schwinger, suppose that one uses the {\it operator} approximation (\ref{eq:backprim}) for $p$ and replaces the right hand side of  (\ref{eq:oper}) with $\langle u_0, \mathcal{P}\uu \rangle$. We get that  
 \begin{multline} \langle u_0, \mathcal{P}\uu \rangle = \int (V_0Q_0)(x)  (T-T_0)\int (V_0Q_0 )^*(x') \uu(x') dx' \bar{u}_0 (x)  dx
\\ =  (b_0^*M_0^{-1}b_0 )^{1/2}
( b^*M^{-1}b )^{1/2} \int\int
e_1^* (T_0+\lambda_j I)^{-1}  \cdot \\ \cdot (V_0Q_0)^*(x){(V_0Q_0)}(x)(T-T_0)(V_0Q_0)^*(x') (V_0Q_0) (x') (T+\lambda_j I)^{-1}e_1dx'dx \\
=  (b_0^*M_0^{-1}b_0 )^{1/2}
( b^*M^{-1}b )^{1/2} e_1^*(T_0+\lambda_j I)^{-1}(T-T_0)(T+\lambda_j I)^{-1}e_1\label{rhsplugin}\end{multline} 
from orthogonality of the basis. At the same time, the left hand side  of (\ref{eq:oper}) is \begin{eqnarray} F_0(\lambda_j)-F(\lambda_j) &=& g^T(\Delta+\lambda_j I)^{-1}g-g^T(\Delta+{p}I+\lambda_j I)^{-1}g \nonumber \\ &=& 
{b_0^*M_0^{-1}b_0} e_1^*(T_0+\lambda_j I)^{-1}e_1-{b^*M^{-1}b} e_1^* (T +\lambda_j I)^{-1}e_1\nonumber\end{eqnarray}
since the ROM matches the data exactly. By definition, $b_0^*M_0^{-1}b_0$ and $b^*M^{-1}b$ are the norms of projections of $g$ on $V_0$ and $V$ respectively, which should become close as the dimension grows, that is, we expect that %Conjecture~\ref{conj:1}  
 \begin{equation} \label{blim} \lim_{m\to\infty} b_0^*M_0^{-1}b_0 =b^*M^{-1}b. \end{equation}  We then obtain 
\begin{eqnarray}\label{eq:ROMint} F_0(\lambda_j)-F(\lambda_j)  & \approx& 
( b^*M^{-1}b ) e_1^*(T_0+\lambda_j I)^{-1}(T-T_0)(T+\lambda_j I)^{-1}e_1 \\
&\approx& \langle u_0, \mathcal{P}\uu \rangle \end{eqnarray}
%\approx F_0(\lambda_j)- F(\lambda_j) \label{eq:identity}\end{multline}
%Suppose now that one puts the back projection approximation for $q$ into the right hand side of (\ref{eq:intd}), we get that 
% \begin{multline}\langle u_0,V_0Q_0(T-T_0)(V_0Q_0)^*u\rangle  \\ =\langle (\Delta+\lambda_j I)^{-1}g, V_0Q_0(T-T_0)(V_0Q_0)^*(\Delta+{q}I+\lambda_j I)^{-1}g \rangle
%\\ \approx (b_0^*M_0^{-1}b_0) \langle 
%V_0Q_0(T_0+\lambda_j I)^{-1}e_1, V_0Q_0(T-T_0)(V_0Q_0)^*V_0Q_0(T+\lambda_j I)^{-1}e_1 \rangle\\
%= (b_0^*M_0^{-1}b_0) e_1^*(T_0+\lambda_j I)^{-1}(T-T_0)(T+\lambda_j I)^{-1}e_1\\ \approx F_0(\lambda_j)- F(\lambda_j) \label{eq:identity}\end{multline}
from  (\ref{rhsplugin}) and (\ref{blim}). Thus we have that, modulo (\ref{blim}), (\ref{eq:backprim}) satisfies the part of system (\ref{eq:oper}) corresponding to $F_0(\lambda_j)-F(\lambda_j)$ for  $j=1,\ldots, m$. Hence the system (\ref{eq:oper}) would be valid for $\lambda_j,\bar\lambda_j$, $j=1,\ldots,m$ if we replace $p$ with its operator approximation (\ref{eq:backprim}). The extension to the derivative data in the real case can be obtained via limiting transition. In summary, the operator $\mathcal{P}$ essentially satisfies the Lippmann-Schwinger-Lanczos equation. 
  
  The back projection algorithm extracts a scalar version of (\ref{eq:backprim}) in a natural way.  Indeed, the simplest back projection reconstruction step is to take 
  	\begin{multline} \label{eq:backprimMark}
  	p(x)= \int p(x')\delta(x-x') d x' \approx \int \mathcal{P} \delta (x-x') d x '  \\ =\int V_0(x') {Q}_0 (T-T_0)  (V_0(x) {Q}_0)^* d x' .\end{multline}
  		It can be show that the probing function (\ref{eq:backprimMark}) uses approximations of $\delta(x-x')$  on the solution snapshot space. Thanks to Conjecture~\ref{conj:1}, this can be defined via projection  
  \[\tilde \delta(x-x')= (V_0(x) {Q}_0) (V_0(x') {Q}_0)^*,\] 
  and is related to  a  known approach of imaging  diagonal operators, e.g., \cite{levinson2016solution}.
   
  The back-projection approach of \cite{druskin2018nonlinear} instead uses a sharper point spread function, obtained by the squaring and re-normalization of the delta function:
 \[ \frac{\tilde \delta(x-x')^2 }{\sqrt{\int  \tilde \delta(x-x')^2 dx'}},\] which yields 
 \be \label{eq:backprim3}
 p(x)\approx \frac{\int \tilde \delta(x-x') \mathcal{P} \tilde  \delta(x-x') dx'}{\int  \tilde \delta(x-x')^2 dx'}  =\frac{V_0(x) {Q}_0 (T-T_0)   (V_0(x) {Q}_0)^*}{\int  \tilde \delta(x-x')^2 dx'}.\ee

\subsection{Linear mapping from ROM to potential}
From (\ref{rhsplugin}), (\ref{eq:ROMint}) and Conjecture~\ref{conj:1} we obtain
\begin{equation}\label{eq:identity2}
s_0(\lambda_j)^*(T-T_0)s(\lambda_j) \approx F_0(\lambda_j)- F(\lambda_j)  \approx
\int \bar{u}_0(x,\lambda_j)\uu(x,\lambda_j)p(x)dx ,
\end{equation}
where $$s_0(\lambda_j)=(b_0^*M_0^{-1}b_0)^{1/2} e_1^*(T_0+\lambda_j I)^{-1}\in\CC^m,$$
  $$s(\lambda_j)=(b^*M^{-1}b)^{1/2} e_1^*(T+\lambda_j I)^{-1}\in\CC^m,$$ and where $\uu$ is given by (\ref{eq:intern}).  Note that $s$ and $s_0$ are data driven. 
\begin{theorem} Assuming that $s,s_0$ and $\uu$ are precomputed from the data, (\ref{eq:identity2})
yields a linear relationship between $p$ and  $(T-T_0)$ that uniquely defines the latter for complex $\lambda_j$, $j=1,\ldots,m$.
For real $\lambda_j$ for uniqueness we need to add derivatives at $\lambda_j$
\end{theorem}
\begin{proof}
The result follows from uniqueness of [m-1/m] Pad\'e from $2m$ matching conditions.
\end{proof}
\subsection{Extension to general ROMs}
The linear mapping approach is not limited to the ROM obtained via frequency matching.
For any reduced order model we can assume the relationship
\be\label{eq:entiredomain}
s_0(\lambda)^*(T-T_0)s(\lambda)
\approx (F_0-F)(\lambda)\approx\\
\int \bar{u}_0(x,\lambda)\uu(x,\lambda)p(x)dx ,
\ee
 for any $ \lambda\in \Sigma \subset \mathbb{C}\setminus \RR_-$, where $\Sigma$ is a compact complex domain.
Then $(T-T_0)$ can be uniquely defined via matching (\ref{eq:entiredomain}) for any $m$ complex $\lambda_j$ in that domain, or $m$ real frequencies and their derivatives.
Furthermore, (\ref{eq:entiredomain}) can be transformed to the time domain where one will obtain convolution equations.

\thanks{{\bf Acknowledgements} 
	{We are grateful Liliana Borcea, Alex Mamonov and J\"orn Zimmerling for productive discussions that inspired this research. }
V. Druskin was partially supported by AFOSR grant FA 955020-1-0079. S. Moskow was partially supported by NSF grants DMS-1715425 and DMS-2008441.  

\bibliography{biblio,biblio6,graphbib6,galerkincitations}
\end{document}